\documentclass[11pt]{article}
\usepackage[utf8]{inputenc}
\usepackage{amsmath, amssymb, amsthm}
\usepackage{lmodern}
\usepackage{geometry}
\usepackage[most]{tcolorbox}
\usepackage{titlesec}
\usepackage[colorlinks=true, linkcolor=red, citecolor=red, urlcolor=cyan]{hyperref}

\definecolor{red1}{RGB}{255,230,230}
\definecolor{red2}{RGB}{255,200,200}
\definecolor{red3}{RGB}{255,170,170}
\definecolor{red4}{RGB}{255,140,140}
\definecolor{red5}{RGB}{255,110,110}

\definecolor{darkred}{RGB}{120, 20, 20}
\titleformat{\section}{\color{darkred}\normalfont\Large\bfseries}{\thesection}{1em}{}

\newtcbtheorem[auto counter, number within=section]{theorem}{Theorem}{
  colback=red5!10!white, colframe=red5!80!black,
  coltitle=white, fonttitle=\bfseries,
  title=Theorem~\thetcbcounter,
  boxed title style={toprule=2pt, colback=red5!70!black},
  breakable, enhanced
}{th}

\newtcbtheorem[auto counter, number within=section]{definition}{Definition}{
  colframe=white, colback=white,
  coltitle=darkred, fonttitle=\bfseries,
  boxed title style={empty},
  breakable, enhanced, sharp corners,
  before skip=10pt, after skip=10pt,
  title={\textcolor{darkred}{Definition~\thetcbcounter}}
}{def}

\newtcbtheorem[auto counter, number within=section]{remark}{Remark}{
  colframe=white, colback=white,
  coltitle=darkred, fonttitle=\bfseries,
  boxed title style={empty},
  breakable, enhanced, sharp corners,
  before skip=10pt, after skip=10pt
}{rmk}

\newtcbtheorem[auto counter, number within=section]{corollary}{Corollary}{
  colframe=white, colback=white,
  coltitle=darkred, fonttitle=\bfseries,
  boxed title style={empty},
  breakable, enhanced, sharp corners,
  before skip=10pt, after skip=10pt
}{cor}

\newtcbtheorem[auto counter, number within=section]{proposition}{Proposition}{
  colframe=white, colback=white,
  coltitle=darkred, fonttitle=\bfseries,
  boxed title style={empty},
  breakable, enhanced, sharp corners,
  before skip=10pt, after skip=10pt
}{prop}

\usepackage{fancyhdr}
\title{\textbf{\textcolor{darkred}{Stochastic Processes and Diffusion Equations}}}

\date{}

\begin{document}
\author{Helder Rojas \\
\small Departamento de Matemáticas Fundamentales\\
\small Universidad Nacional de Educación a Distancia, España\\
}
\maketitle
\thispagestyle{fancy}

In these lecture notes, we explore the mathematical preliminaries and foundational concepts that connect stochastic processes with partial differential equations. We begin by investigating Brownian motion, which serves as a model for random fluctuations and is deeply connected to the heat equation. This connection forms the basis for understanding diffusion phenomena, where the probability distribution of Brownian motion evolves according to the heat equation over time. To extend this classical result to more general stochastic systems, we introduce the Itô calculus, a powerful framework that allows us to analyze processes driven by both deterministic drift and stochastic fluctuations. This mathematical tool is essential for understanding the dynamics of more complex diffusion processes, where randomness is no longer purely Brownian, but also depends on the underlying system’s state. Building on these concepts, we turn to the study of diffusion processes, which generalize Brownian motion by incorporating the Fokker-Planck equation. This equation describes how the probability density of a diffusion process evolves over time and serves as an extension of the heat equation to more complex stochastic systems. By using Itô calculus, we can rigorously study how these processes behave and connect the microscopic randomness of individual particles to their macroscopic description via partial differential equations. The results presented in this section include theorems, lemmas, and corollaries, whose proofs are omitted for brevity. These can be found in classical references on stochastic analysis, such as \cite{Protter2005, oksendal, kuo, Kallenberg2021, Rogers_Williams_2000}.

\section{Brownian motion and the Heat equation}
Brownian motion ($W_t$) is a fundamental stochastic process that serves as the cornerstone of modern probability theory and stochastic analysis. Originally introduced as a mathematical model for random particle displacement in a fluid, it has since become a critical tool in diverse fields, including financial mathematics, statistical physics, and partial differential equations (PDEs). From a probabilistic perspective, Brownian motion is the canonical continuous-time Markov process, characterized by independent and normally distributed increments. These properties, along with its role as a martingale, make it a natural model for stochastic systems exhibiting purely random fluctuations. Furthermore, its transition densities satisfy the heat equation, establishing a deep connection between probability theory and classical PDEs. This relationship is crucial in the study of diffusion processes and potential theory.

In this section, we present the formal definition of Brownian motion, its transition density function, and its existence theorem. Finally, we state a fundamental result linking Brownian motion to the heat equation, illustrating how the evolution of its probability distribution is governed by a deterministic PDE.

\newpage

\begin{definition}{}{}
 A stochastic process \(\{W_t : t \geq 0\}\) defined on a probability space \((\Omega, \mathcal{F}, \mathbb{P})\) is called a \textit{standard Brownian motion} (or standard Wiener process) if it satisfies the following properties:   
\begin{enumerate}
    \item \(W_0 = 0\) almost surely.
    \item For \(0 \leq s < t\), the increment \(W_t - W_s \sim \mathcal{N}(0, t-s)\), where \(\mathcal{N}(0, t-s)\) denotes a normal distribution with mean \(0\) and variance \(t-s\).
    \item The increments \(W_{t_i} - W_{t_{i-1}}\) are independent for \(t_0 < t_1 < \cdots < t_n\).
    \item The sample paths \(t \mapsto W_t\) are continuous almost surely.
\end{enumerate}
\end{definition}

\begin{remark}{}{}
The Brownian motion \( W_t \) described above starts at \( 0 \). In certain situations, we may require a Brownian motion that begins at a different starting point \( x \). Such a process is represented as \( x + W_t \). When the starting point is not \( 0 \), it will be explicitly specified as \( x \).
\end{remark}

\begin{proposition}{}{}
 Let \(\{W_t\}_{t \geq 0}\) be a Brownian motion defined starting at $x$ on a probability space \((\Omega, \mathcal{F}, \mathbb{P})\), taking values in \(\mathbb{R}\). Then, for any \(t > 0\), the transition density \(p(t, x, y)\) of the process is given by
 \begin{equation}\label{td}
    p(t, x, y) = \frac{1}{\sqrt{2 \pi t}} \exp\left(-\frac{(y - x)^2}{2t}\right), \quad x, y \in \mathbb{R}. 
 \end{equation}  
\end{proposition}
    
\begin{theorem}{Existence of Brownian motion}{exis-BM}
There exists a probability space \((\Omega, \mathcal{F}, \mathbb{P})\) and a stochastic process \(\{W_t\}_{t \geq 0}\) on \(\Omega\) such that the finite-dimensional distributions 
of \(W_t\) are given by
\begin{equation}\label{dfd}
\mathbb{P}(W_{t_1} \in F_1, \cdots, W_{t_k} \in F_k) =
\int_{F_1 \times \cdots \times F_k} p(t_1, x, x_1) \cdots p(t_k - t_{k-1}, x_{k-1}, x_k) \, dx_1 \cdots dx_k,
\end{equation}
where \(F_1, \cdots, F_k\) denote Borel sets in \(\mathbb{R}\) and \(k \in \mathbb{N}\).
\end{theorem}

\begin{proof}
If $0 \leq t_1 \leq t_2 \leq \cdots \leq t_k$, define a measure $\nu_{t_1, \ldots, t_k}$ on $\mathbb{R}^{k}$ by
\[
\nu_{t_1, \ldots, t_k}(F_1 \times \cdots \times F_k) :=
\int_{F_1 \times \cdots \times F_k} p(t_1, x, x_1) p(t_2 - t_1, x_1, x_2) \cdots p(t_k - t_{k-1}, x_{k-1}, x_k) \, dx_1 \cdots dx_k,
\]
where $p(0, x, y)\, dy_1 \cdots dy_k = \delta_x(y)$, the unit point mass at $x$. Since $\int_{\mathbb{R}} p(t, x, y) \, dy = 1$ for all $t\geq 0$, the probability measure $\nu_{t_1, \ldots, t_k}$ satisfies the consistency conditions of Kolmogorov’s extension theorem, see \cite{oksendal}. Therefore, there is a stochastic process $\{W_t\}_{t \geq 0}$ on $\Omega$ whose transition density is given by \eqref{td}.
    
\end{proof}

\begin{figure}[!h]
    \centering
    \includegraphics[width=0.55\linewidth]{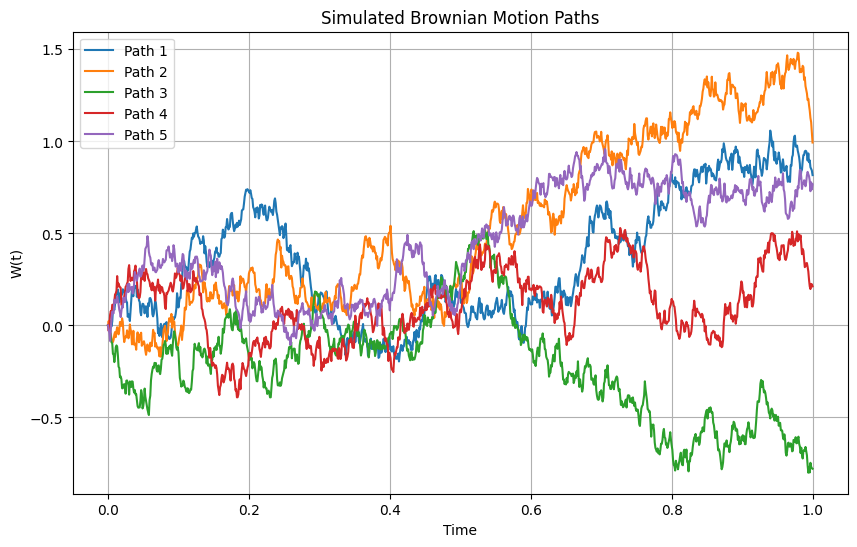}
    \caption[Simulated Brownian motion trajectories.]{This plot shows five simulated trajectories of Brownian motion. Each line represents an independent realization of the stochastic process.}
    \label{fig:BM}
\end{figure}

\begin{theorem}{Connection between Brownian motion and the Heat equation}{Conexion-one-dim}
Let \(\{W_t\}_{t \geq 0}\) be a Brownian motion starting at \(x\) defined on a probability space \((\Omega, \mathcal{F}, \mathbb{P})\), and assume that \(f \in C(\mathbb{R})\cap L^{\infty}( \mathbb{R})\). Furthermore, consider the heat equation
\begin{equation}\label{HeatEquation}
   \frac{\partial u}{\partial t} = \frac{\partial^2 u}{\partial x^2}, \quad x \in \mathbb{R}, \, t > 0,
\end{equation}
with initial condition \(u(x, 0) = f(x)\). Then, it holds that:
\begin{enumerate}
    \item There exists a classical solution \(u \in C^{\infty}(\mathbb{R} \times (0, \infty))\) given by
    \[
    u(x, t) = \frac{1}{\sqrt{4\pi t}} \int_{\mathbb{R}} f(y) \exp\left(-\frac{(y - x)^2}{4t}\right) \, dy, \quad t > 0.
    \]
    \item The solution \(u(x, t)\) admits a stochastic representation in terms of Brownian motion:
    \[
    u(x, t) = \mathbb{E}[f(W_t) \mid W_0 = x],
    \]
    where \(\mathbb{E}\) denotes the expectation with respect to the probability measure induced by the Brownian motion \(\{W_t\}_{t \geq 0}\).
\end{enumerate}
\end{theorem}

\begin{proof}
    The fundamental solution of \eqref{HeatEquation} for the initial \(u(x, 0) = \delta(x - x_0)\) is given by normal Kernel
    \begin{equation}\label{kernel}
        \phi(x, t; x_0) = \frac{1}{\sqrt{4\pi t}} \exp\left(-\frac{(x - x_0)^2}{4t}\right),
    \end{equation}
    where \(\delta(x - x_0)\) is the Dirac delta function at \(x_0\). This Kernel satisfies the equation 
    \[
    \frac{\partial \phi}{\partial t} = \frac{\partial^2 \phi}{\partial x^2}, \quad \phi(x, 0; x_0) = \delta(x - x_0).
    \]
    For a general initial condition \(u(x, 0) = f(x)\), the solution is obtained by convolution with the fundamental solution
    \begin{equation}\label{solConv}
    u(x, t) = \int_{\mathbb{R}} \phi(x, t; y) f(y) \, dy.  
    \end{equation}
    Substituting \eqref{kernel} in \eqref{solConv} , we obtain
    \begin{equation}\label{clasical-solution}
     u(x, t) = \frac{1}{\sqrt{4\pi t}} \int_{\mathbb{R}} f(y) \exp\left(-\frac{(x - y)^2}{4t}\right) \, dy.   
    \end{equation}
By direct differentiation, it can be verified that \(u(x, t)\) satisfies \eqref{HeatEquation}. The initial condition \(u(x, 0) = f(x)\) holds due to the dominated convergence theorem. Therefore, there exists a classical solution given by \eqref{clasical-solution}, for more details see \cite{evans2010pde}.

On the other hand, from Theorem \ref{th:exis-BM}
, we know that there is a stochastic process, a Brownian motion starting at $x$, which we denote by $\{W_t\}_{t \geq 0}$, whose transition density is given by \eqref{kernel}. Therefore, the expected value of $f(W_t)$ given $W_0 = x$ is
\begin{equation}\label{expec-BM}
    \mathbb{E}[f(W_t) \mid W_0 = x] = \int_{\mathbb{R}} f(y) \phi(y, t; x) \, dy.
\end{equation}
From \eqref{kernel}, \eqref{clasical-solution} and \eqref{expec-BM}, we obtain the representation $u(x, t) = \mathbb{E}[f(W_t) \mid W_0 = x]$, which concludes the proof.
\end{proof}

In order to generalize to arbitrary dimensions the result presented in Theorem \ref{th:Conexion-one-dim}, we need to establish a definition of multidimensional Brownian motion.

\begin{definition}{Multidimensional Brownian motion}{}
   Let \(d \in \mathbb{N}\), and consider a probability space \((\Omega, \mathcal{F}, \mathbb{P})\). A \(d\)-dimensional Brownian motion is a stochastic process \(\{B_t\}_{t \geq 0} = \{(B_t^1, B_t^2, \dots, B_t^d)\}_{t \geq 0}\), where each component \(B_t^i\) (\(i = 1, \dots, d\)) is a Brownian motion defined on the probability space \((\Omega, \mathcal{F}, \mathbb{P})\).
\end{definition}

\begin{figure}[htp]
    \centering
    \includegraphics[width=0.6\textwidth, height=0.25\textheight]{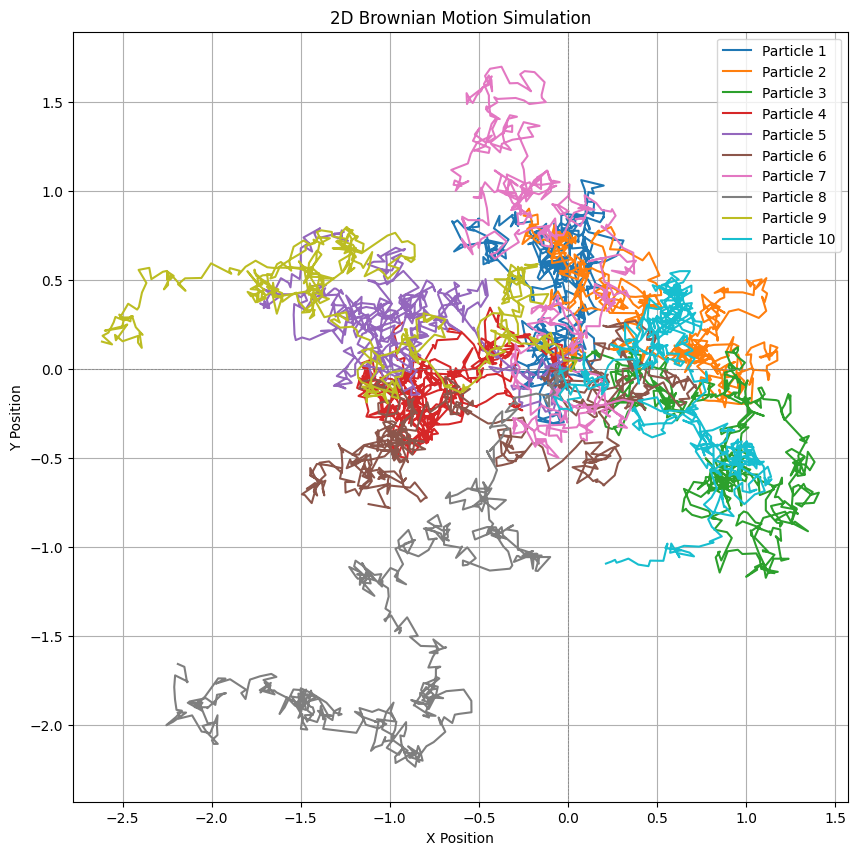}
    \caption[2D Brownian motion simulation showing independent particle trajectories in $x$ and $y$ directions.]{Here is a simulation of 2D Brownian motion. Each line represents the trajectory of a particle undergoing random motion over time in a two-dimensional space. The trajectories illustrate the stochastic nature of the process, with independent movements in the $x$ and $y$ directions.}
    \label{fig:2D-BM}
\end{figure}

\textbf{}

\newpage

\begin{theorem}{}{conexion-multi-dim}
Let \(\{W_t\}_{t \geq 0}\) be a \(d\)-dimensional Brownian motion starting at $x$ defined on a probability space \((\Omega, \mathcal{F}, \mathbb{P})\), and and assume that \(f \in C(\mathbb{R}^d)\cap L^{\infty}( \mathbb{R}^d)\). Furthermore, Consider the \(d\)-dimensional heat equation
\[
\frac{\partial u}{\partial t} = \Delta u, \quad x \in \mathbb{R}^d, \, t > 0,
\]
with initial condition \(u(x, 0) = f(x)\), where \(\Delta = \sum_{i=1}^d \frac{\partial^2}{\partial x_i^2}\) is the Laplace operator. Then:

\begin{enumerate}
    \item There exists a classical solution \(u\in C^{\infty}(\mathbb{R}^d\times(0,\infty))\) given by
    \[
    u(x, t) = \frac{1}{(4\pi t)^{d/2}} \int_{\mathbb{R}^d} f(y) \exp\left(-\frac{\|x - y\|^2}{4t}\right) \, dy, \quad t > 0.
    \]

    \item The solution \(u(x, t)\) admits a stochastic representation in terms of \(d\)-dimensional Brownian motion
    \[
    u(x, t) = \mathbb{E}[f(W_t) \mid W_0 = x],
    \]
    where \(\mathbb{E}\) denotes expectation with respect to the probability measure induced by the Brownian motion \(\{W_t\}_{t \geq 0}\).
\end{enumerate}   
\end{theorem}

Therefore, the result presented in Theorem \ref{th:conexion-multi-dim} is the multidimensional generalization of the connection between Brownian motion and the heat equation presented in Theorem \ref{th:Conexion-one-dim}. Furthermore, from this result follows the interesting, and historically important, relation between the density function of Brownian motion and the heat equation. In the following Lemma we establish this relation.

\begin{corollary}{}{BrownianHeatDensity}
Let \(\{W_t\}_{t \geq 0}\) be a \(d\)-dimensional Brownian motion starting at $x$, then its transition density $p_x(t, y)$ satisfies the equation 
\begin{equation}\label{relation-BM}
 \frac{\partial p}{\partial t} = \frac{1}{2}\,\Delta p, \quad x \in \mathbb{R}^d, \, t > 0,   
\end{equation}
with initial condition $p(0,x, y)=\delta(y-x).$ 
\end{corollary}

This results reveals a scale-dependent interpretation of Brownian motion. At the microscopic level, the stochastic trajectories describe individual particle displacements influenced by local randomness. However, at the macroscopic level, the probability density function evolves smoothly according to the deterministic heat equation, which governs the diffusion of particle concentration over time. The heat equation thus emerges as the macroscopic law that describes the large-scale behavior of a stochastic system whose local dynamics are driven by Brownian motion. This connection is fundamental in stochastic analysis and establishes a bridge between probability theory and PDEs, playing a crucial role in diffusion processes, statistical physics, and potential theory. 

\begin{figure}[htbp]
    \centering
    \includegraphics[width=0.95\textwidth, height=0.25\textheight]{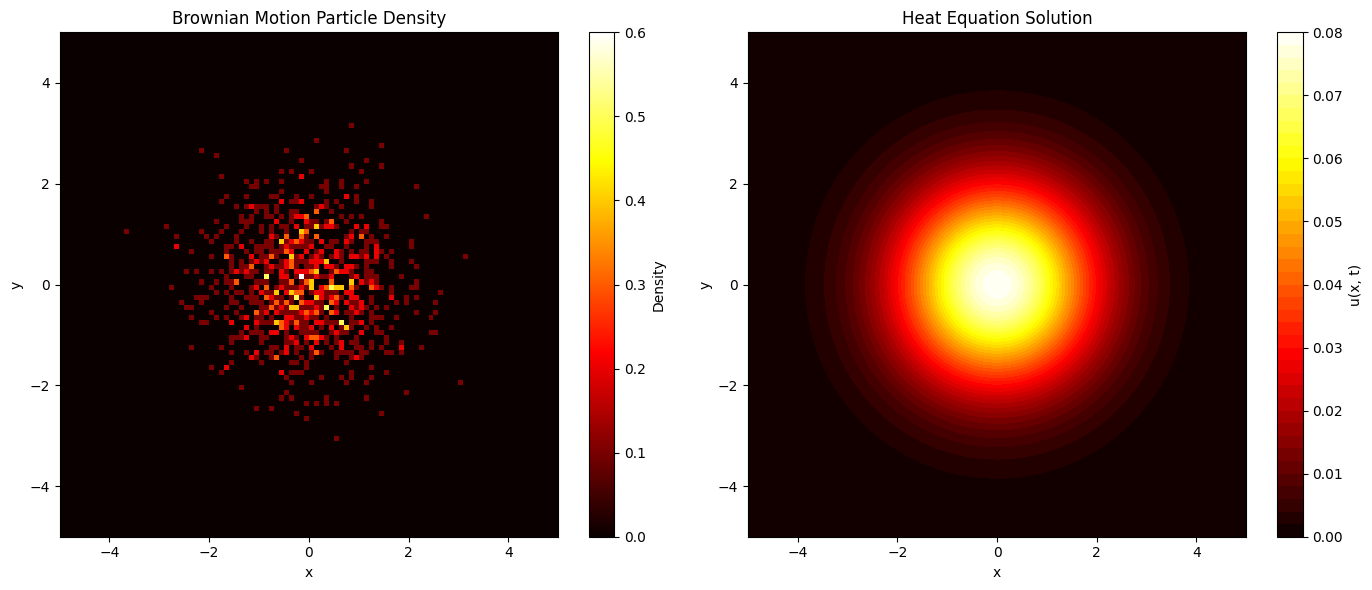}
    \caption[Comparison of 2D Brownian motion density and heat equation solution, both showing Gaussian-like distributions centered at the origin.]{The simulation compares 2D Brownian motion particle density (left) and the 2D heat equation solution (right), showing their equivalence as both form Gaussian-like distributions centered at the origin.}
    \label{fig:Conexion-MB-HE}
\end{figure}

We numerically illustrate the relationship, in Figure \ref{fig:Conexion-MB-HE} compares the empirical density of a two-dimensional Brownian motion (left) with the numerical solution of the corresponding heat equation (right). The empirical density is estimated from the trajectories of $N$ independent Brownian particles using a two-dimensional histogram, while the theoretical density is obtained by solving \eqref{relation-BM}. Both densities exhibit a Gaussian-like profile centered at the origin, validating the well-known result that the heat equation governs the macroscopic evolution of Brownian motion probability densities.

While Brownian motion serves as the fundamental model for stochastic dynamics, many natural and applied systems exhibit structured randomness beyond pure Brownian fluctuations. In particular, real-world stochastic processes often incorporate {deterministic drift}, governing systematic trends, and {state-dependent volatility}, modulating stochastic fluctuations. A natural extension is given by \textit{diffusion processes}, which generalize Brownian motion by satisfying the stochastic differential equation (SDE)
\[
dX_t = b(X_t)\,\mathrm{d}t + \sigma(X_t) \,\mathrm{d}W_t,
\]
where \( b(X_t) \) represents the deterministic drift, and \( \sigma(X_t) \) is the diffusion coefficient controlling local variability. These processes arise in diverse applications, from {physics} (particle transport, thermodynamics) to {finance} (stochastic modeling of asset prices) and {biology} (population dynamics, neural activity).   The rigorous mathematical formulation of diffusion processes requires a well-defined framework for {stochastic integration and differentiation}, which extends classical calculus to non-differentiable paths. This leads to the development of \textit{Itô calculus}, a fundamental tool for analyzing stochastic dynamics. The next section develops the \textit{Itô integral} and the \textit{Itô formula} and their role in the theory of stochastic differential equations, fundamental tools in the study of stochastic processes.

\section{Itô Calculus}

Building on the foundation of Brownian motion and its connection to the heat equation, we now transition to the study of Itô calculus. Itô calculus provides the necessary framework for rigorously defining stochastic integrals and differential equations that describe the evolution of processes with both deterministic drift and stochastic components. 

\vspace{0.2cm}

  Let \((\Omega, \mathcal{F}, \mathbb{P})\) be a probability space equipped with a filtration \((\mathcal{F}_t)_{t \geq 0}\) satisfying the usual conditions (increasing, right-continuous, and complete). Let \((W_t)_{t \geq 0}\) be a standard Brownian motion adapted to \((\mathcal{F}_t)\). 
\begin{definition}{Stochastic Integral}{}
   For a process \(f: [0, T] \times \Omega \to \mathbb{R}\) that is \(\mathcal{F}_t\)-adapted and belongs to the space \(L^2([0, T] \times \Omega)\), i.e.,

\[
\mathbb{E} \left[ \int_0^T |f(t, \omega)|^2 \, \mathrm{d}t \right] < \infty,
\]
the \textbf{Itô integral} of \(f\) with respect to \(W_t\) over \([0, T]\) is defined as the limit of stochastic sums
\[
I(f) = \int_0^T f(t) \, \mathrm{d}W_t := \lim_{|\mathcal{P}| \to 0} \sum_{i=0}^{n-1} f(t_i)(W_{t_{i+1}} - W_{t_i}),\quad \textrm{in probability},
\]

where \(\mathcal{P} = \{0 = t_0 < t_1 < \cdots < t_n = T\}\) is a partition of \([0, T]\), \(|\mathcal{P}|\) is the maximum size of the subintervals \([t_i, t_{i+1}]\), and the evaluation points \(t_i\) satisfy \(t_i \in [t_i, t_{i+1})\).
\end{definition}

\begin{theorem}{Fundamental Properties of the Itô Integral}{}
Let \(f, g \in L^2([0, T] \times \Omega)\) be processes adapted to \(\mathcal{F}_t\), and let \(a, b \in \mathbb{R}\). Then

   \begin{enumerate}
    \item \textbf{Linearity.} The Itô integral is linear, i.e., 
    \[
    \int_0^T \big(a f(t) + b g(t)\big) \, \mathrm{d}W_t = a \int_0^T f(t) \, \mathrm{d}W_t + b \int_0^T g(t) \, \mathrm{d}W_t.
    \]

    \item \textbf{Isometry.} The Itô integral satisfies the isometry property, i.e., 
    \[
    \mathbb{E}\left[\left(\int_0^T f(t) \, \mathrm{d}W_t\right)^2\right] = \mathbb{E}\left[\int_0^T f(t)^2 \, \mathrm{d}t\right].
    \]

    \item \textbf{Martingale Property.} The process
    \[
    M_t = \int_0^t f(s) \, \mathrm{d}W_s, \quad t \in [0, T],
    \]
    is a martingale with respect to the filtration \((\mathcal{F}_t)\).

    \item \textbf{Quadratic Variation.} The quadratic variation of the Itô integral is given by
    \[
    \left[\int_0^t f(s) \, \mathrm{d}W_s\right]_t = \int_0^t f(s)^2 \, \mathrm{d}s.
    \]

    \item \textbf{Zero-Mean Property.} The expectation of the Itô integral is zero, i.e., 
    \[
    \mathbb{E}\left[\int_0^T f(t) \, \mathrm{d}W_t\right] = 0.
    \]
\end{enumerate}
\end{theorem}

Subsequently, we present the definition of the multi-dimensional Itô integral as follows:

\begin{definition}{Multi-dimensional Itô Integral}{}
 let $(W_t)_{t \geq 0}$ be a standard $d$-dimensional Brownian motion. Furthermore, let $f(t, \omega): [0, T] \times \Omega \to \mathbb{R}^{n \times d}$ be a matrix of $\mathcal{F}_t$-adapted processes such that each entry $f_{ij}(t, \omega)$ belongs to the space $L^2([0, T] \times \Omega)$, i.e.,
 \[
\mathbb{E}\left[ \int_0^T \| f(t) \|_F^2 \, \mathrm{d}t \right] < \infty,
\]
where $\| f(t) \|_F$ denotes the Frobenius norm of the matrix $f(t)$. Then the \textbf{multi-dimensional Itô integral} is defined as:
\[
\int_0^T f(t) \, \mathrm{d}W_t = 
\begin{pmatrix}
\sum\limits_{j=1}^d \int\limits_0^T f_{1j}(t) \, \mathrm{d}W_t^j, \dots, 
\sum\limits_{j=1}^d \int\limits_0^T f_{nj}(t) \, \mathrm{d}W_t^j
\end{pmatrix}^\top,
\]
where each term is a sum of scalar Itô integrals.
\end{definition}

Brownian motion serves as the fundamental model for continuous-time stochastic processes, but many real-world phenomena exhibit structured randomness beyond its purely diffusive nature. Itô processes generalize Brownian motion by incorporating a deterministic drift term and a state-dependent diffusion coefficient. The rigorous definition of Itô processes relies on Itô integration, which provides a well-defined framework for handling stochastic integrals. In this section, we present the formal definition of an Itô process and establish its properties within the broader framework of stochastic calculus.

\begin{definition}{Itô Process}{}
    Let \((\Omega, \mathcal{F}, \{\mathcal{F}_t\}_{t \geq 0}, \mathbb{P})\) be a filtered probability space and  let \(W_t \) be a standard \(m\)-dimensional Brownian motion adapted to the filtration \(\{\mathcal{F}_t\}_{t \geq 0}\). A vector-valued stochastic process \(\{X_t\}_{t \geq 0} \in \mathbb{R}^d\), defined on \((\Omega, \mathcal{F}, \mathbb{P})\), is called a \textbf{$d$-dimensional Itô process} if it can be expressed as
    \begin{equation}\label{IP}
      X_t = X_0 + \int_0^t \mu(s, X_s) \, \mathrm{d}s + \int_0^t \sigma(s, X_s) \, \mathrm{d}W_s, \quad t \geq 0,  
    \end{equation}
for some \(\mathcal{F}_0\)-measurable random vector \(X_0 \in \mathbb{R}^d\), and for some functions \(\mu: \mathbb{R}_+ \times \mathbb{R}^d \to \mathbb{R}^d\) (called the \textit{drift vector}) and \(\sigma: \mathbb{R}_+ \times \mathbb{R}^d \to \mathbb{R}^{d \times m}\) (called the \textit{diffusion matrix}) such that all the expressions are well-defined, i.e., measurable functions such that
$$\int_0^T \|\mu(s, X_s)\| \, \mathrm{d}s < \infty \quad \textrm{and} \quad \int_0^T \|\sigma(s, X_s)\|_F^2 \, \mathrm{d}s < \infty,\,\,\, \mathbb{P}-\textbf{a.s.}, \,\,\,\textrm{for all}\,\,\, T > 0.$$
\end{definition}
\

If \(\{X_t\}_{t \geq 0}\) is an Itô Process of the form \eqref{IP}  is sometimes written in the shorter differential form, Stochastic Differential Equation (SDE), given by 
\[
\mathrm{d}X_t = \mu(t, X_t) \, \mathrm{d}t + \sigma(t, X_t) \, \mathrm{d}W_t.
\]

\newpage

\begin{theorem}{Itô's Formula}{}
    Let \(\{X_t\}_{t \geq 0}\) be a \(d\)-dimensional Itô process given by
\[
\mathrm{d}X_t = \mu(t, X_t) \, \mathrm{d}t + \sigma(t, X_t) \, \mathrm{d}W_t,
\]
Let \(f(t,x)=(f_1(t,x),\cdots f_p(t,x))\in C^{2}([0,\infty)\times\mathbb{R}^d,\,\mathbb{R}^p)\), i.e., twice continuously differentiable on $[0,\infty)\times\mathbb{R}^d)$. Then $Y_t=f(t, X_t)$ is again an Itô process, and
\[
\mathrm{d}Y_t
= \frac{\partial f}{\partial t}(t, X_t) \, \mathrm{d}t 
+ \nabla f(t, X_t)^\top \, \mathrm{d}X_t 
+ \frac{1}{2} \operatorname{Tr}\left( \sigma(t, X_t) \sigma(t, X_t)^\top \nabla^2 f(t, X_t) \right) \, \mathrm{d}t,
\]
where \(\nabla f(t, X_t) \in \mathbb{R}^d\) and \(\nabla^2 f(t, X_t) \in \mathbb{R}^{d \times d}\) is Gradient and Hessian matrix of \(f\) with respect to \(x\), respectively. Therefore, expanding \(\mathrm{d}X_t\) gives
\[
\mathrm{d}Y_t = \left[\frac{\partial f}{\partial t} + \nabla f^\top \mu + \frac{1}{2} \operatorname{Tr}\left(\sigma \sigma^\top \nabla^2 f\right)\right] \, \mathrm{d}t
+ \nabla f^\top \sigma \, \mathrm{d}W_t.
\]
\end{theorem}

\begin{theorem}{Itô's Integration by Parts}{}
Let \(\{X_t\}_{t \geq 0}\) and \(\{Y_t\}_{t \geq 0}\) be two Itô processes defined by
\[
\mathrm{d}X_t = \mu_X(t, X_t) \, \mathrm{d}t + \sigma_X(t, X_t) \, \mathrm{d}W_t, \quad
\mathrm{d}Y_t = \mu_Y(t, Y_t) \, \mathrm{d}t + \sigma_Y(t, Y_t) \, \mathrm{d}W_t,
\]
where \(W_t\) is a Brownian motion. Then, the product \(Z_t = X_t Y_t\) satisfies
\[
\mathrm{d}(X_t Y_t) = X_t \, \mathrm{d}Y_t + Y_t \, \mathrm{d}X_t + \mathrm{d}\langle X, Y \rangle_t,
\]
where \(\mathrm{d}\langle X, Y \rangle_t\) is the differential of the quadratic covariation process of \(X_t\) and \(Y_t\), given by
\[
\mathrm{d}\langle X, Y \rangle_t = \sigma_X(t, X_t) \sigma_Y(t, Y_t)^\top \, \mathrm{d}t.
\]
Therefore, expanding the terms explicitly gives
\begin{align*}
\mathrm{d}(X_t Y_t) 
&= \big( X_t \mu_Y(t, Y_t) + Y_t \mu_X(t, X_t) 
+ \sigma_X(t, X_t) \sigma_Y(t, Y_t)^\top \big) \, \mathrm{d}t \\
&\quad + X_t \sigma_Y(t, Y_t) \, \mathrm{d}W_t 
+ Y_t \sigma_X(t, X_t) \, \mathrm{d}W_t.
\end{align*}

\end{theorem}

\newpage

\begin{theorem}{Existence and uniqueness theorem for SDEs}{existence-uniqueness}
Let \((\Omega, \mathcal{F}, \{\mathcal{F}_t\}_{t \geq 0}, \mathbb{P})\) be a filtered probability space, and let \(\{W_t\}_{t \geq 0}\) be a standard \(m\)-dimensional Brownian motion adapted to the filtration \(\{\mathcal{F}_t\}_{t \geq 0}\). Consider the SDE in \(\mathbb{R}^d\)
\[
\mathrm{d}X_t = \mu(t, X_t) \, \mathrm{d}t + \sigma(t, X_t) \, \mathrm{d}W_t, \quad X_0=Z,
\]
where \(Z\) is a random variable, \(\mathcal{F}_0\)-measurable, with \(\mathbb{E}[\|X_0\|^2] < \infty\). Furthermore, assume the following conditions hold:
\begin{itemize}
    \item {Lipschitz continuity:} There exist constants \(L_{\mu}, L_\sigma > 0\) such that for all \(t \in [0, T]\) and \(x, y \in \mathbb{R}^d\)
    \[
    \|\mu(t, x) - \mu(t, y)\| \leq L_{\mu} \|x - y\|, \quad
    \|\sigma(t, x) - \sigma(t, y)\|_F \leq L_\sigma \|x - y\|,
    \]
    where \(\|\cdot\|_F\) denotes the Frobenius norm.
    
    \item {Linear growth:} There exist constants \(K_{\mu}, K_\sigma > 0\) such that for all \(t \in [0, T]\) and \(x \in \mathbb{R}^d\)
    \[
    \|\mu(t, x)\| \leq K_{\mu}(1 + \|x\|), \quad \|\sigma(t, x)\|_F \leq K_\sigma(1 + \|x\|).
    \]
\end{itemize}

Then, there exists a unique stochastic process \(\{X_t\}_{t \geq 0}\) satisfying the SDE such that:
\begin{enumerate}
    \item \textbf{Existence:} The solution \(X_t\) exists, is adapted to \(\{\mathcal{F}_t\}_{t \geq 0}\), and satisfies
    \[
    \mathbb{E} \left[ \sup_{t \in [0, T]} \|X_t\|^2 \right] < \infty.
    \]
    \item \textbf{Uniqueness:} If \(\{X_t\}\) and \(\{Y_t\}\) are two solutions with the same initial condition \(Z\), then
    \[
    \mathbb{P} \left( X_t = Y_t, \, \forall t \in [0, T] \right) = 1.
    \]
\end{enumerate}
\end{theorem}

Itô processes form a broad class of continuous stochastic processes constructed via stochastic integration with respect to Brownian motion. While they capture random evolution in systems with drift and diffusion components, they do not necessarily possess the Markov property or arise as solutions to well-posed stochastic differential equations. \textit{Diffusion processes} refine this notion by requiring that the process be a strong solution to an SDE with measurable coefficients, continuous paths, and the Markov property. Such processes admit an infinitesimal generator, often associated with second-order differential operators, and their transition probabilities typically satisfy the Fokker–Planck or Kolmogorov forward equation. The next chapter introduces the formal definition of diffusion processes and explores their analytical structure within the framework of stochastic differential equations and probabilistic PDE theory.

\section{Diffusion Processes}

Following the foundation laid by Itô calculus, we now turn to the study of diffusion processes, which generalize Brownian motion by incorporating both deterministic drift and state-dependent diffusion. In this section, we define diffusion processes rigorously and explore their key properties, including the associated infinitesimal generator. We will also examine how these processes lead to the formulation of the Fokker-Planck equation, which describes the evolution of the probability density of a diffusion process over time.

Diffusion processes are continuous-time, continuous-path Markov processes that arise as strong solutions to stochastic differential equations driven by Brownian motion. They constitute a fundamental class of stochastic models with deep connections to partial differential equations and infinitesimal generators.
    
\begin{definition}{}{}
 A \(d\)-dimensional stochastic process \(\{X_t\}_{t \geq 0} = (X_t^1, \dots, X_t^d)^\top \in \mathbb{R}^d\) is called an \textbf{Itô diffusion process} if it satisfies the SDE  
 \[
\mathrm{d}X_t = \mu(t, X_t) \, \mathrm{d}t + \sigma(t, X_t) \, \mathrm{d}W_t, \quad X_0 = Z,
\]
where \(\{W_t\}_{t \geq 0}\) be an \(m\)-dimensional Brownian motion and \(\mu: \mathbb{R}_+ \times \mathbb{R}^d \to \mathbb{R}^d\), \(\sigma: \mathbb{R}_+ \times \mathbb{R}^d \to \mathbb{R}^{d \times m}\), $Z\in\mathbb{R}^d$ satisfy the conditions in Theorem \ref{th:existence-uniqueness}. Furthermore, \(\{X_t\}_{t \geq 0}\) is called an \textbf{ (time-homogeneous) Itô diffusion} if $\mu$ and $\sigma$ do not depend explicitly on time, i.e., $\mu(t, X_t)=\mu(X_t)$ and $\sigma(t, X_t)=\sigma(X_t)$.
\end{definition}

A widely used diffusion process in financial and mathematical models is the Geometric Brownian Motion (GBM) \(\{S_t\}_{t \geq 0}\in\mathbb{R}\), described by the one-dimensional SDE
\begin{equation}\label{GBM-SDE}
    \mathrm{d}S_t = \mu S_t \, \mathrm{d}t + \sigma S_t \, \mathrm{d}W_t,\quad, S_0=s,
\end{equation}
where \(\mu\in\mathbb{R}\) and \(\sigma\in\mathbb{R}^+\) are the global growth rate and volatility of the process. This model is commonly applied in financial models, such as the Black-Scholes model for option pricing. In Figure \ref{fig:GBM}, the process $S_t$ grows on average due to the global growth rate, with fluctuations caused by volatility terms. This behavior is characteristic of the evolution of financial asset prices.
\begin{figure}[ht]
    \centering
    \includegraphics[width=0.6\textwidth, height=0.25\textheight]{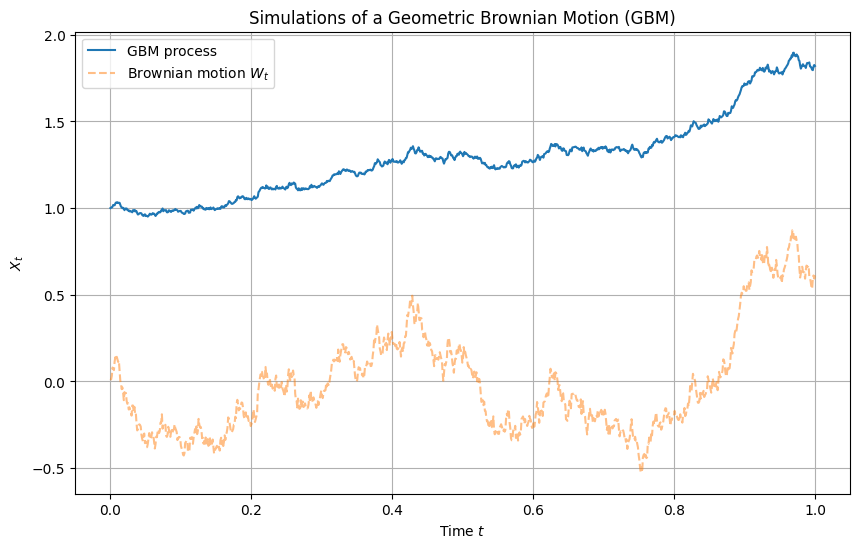}
    \caption{Simulation of $S_t$ described by \eqref{GBM-SDE}, where $\mu=0.5$, $\mu=0.2$ and $s=1$.}
    \label{fig:GBM}
\end{figure}

The Markov and strong Markov properties formalize the memoryless behavior of diffusion processes. The Markov property ensures that conditional future distributions depend only on the present state, while the strong Markov property extends this to random times, ensuring the process restarts probabilistically at stopping times.

\begin{theorem}{The Markov property for Itô diffusions}{}
 Let \((\Omega, \mathcal{F}, \{\mathcal{F}_t\}_{t \geq 0}, \mathbb{P})\) be a filtered probability space and  \(\{X_t\}_{t \geq 0} \in \mathbb{R}^d\) an Itô diffusion defined on \((\Omega, \mathcal{F}, \mathbb{P})\). Then, for all \(0 \leq s \leq t\) and any bounded measurable function \(f: \mathbb{R}^d \to \mathbb{R}\), it holds that
 \[
\mathbb{E}[ f(X_t) \mid \mathcal{F}_s ] = \mathbb{E}[ f(X_t) \mid X_s ], \quad \mathbb{P}-\text{a.s.}
\]
\end{theorem}

\begin{theorem}{The strong Markov property for Itô diffusions}{}
Let \(\{X_t\}_{t \geq 0}\) be an Itô diffusion in \(\mathbb{R}^d\), adapted to a filtration \(\{\mathcal{F}_t\}_{t \geq 0}\), and let \(\tau\) be a \textit{stopping time} with respect to \(\{\mathcal{F}_t\}_{t \geq 0}\), meaning that \(\{\tau \leq t\} \in \mathcal{F}_t\) for all \(t \geq 0\).  
Then, for every stopping time \(\tau\), every \(t \geq 0\), and any bounded measurable function \(f: \mathbb{R}^d \to \mathbb{R}\), we have
\[
\mathbb{E} [ f(X_{t+\tau}) \mid \mathcal{F}_\tau ] = \mathbb{E} [ f(X_{t+\tau}) \mid X_\tau ], \quad \mathbb{P}-\text{a.s.}
\]
\end{theorem}

A fundamental tool in the analysis of diffusion processes is the infinitesimal generator, which encodes the local behavior of the process. It characterizes the first-order evolution of expectations of test functions and plays a central role in the connection between stochastic processes and partial differential equations. The following definition formalizes the generator for a time-homogeneous Itô diffusion. Furthermore, associated with each Itô diffusion, there is a second-order partial differential operator that is precisely the generator of the diffusion.

\begin{definition}{}{}
    Let \(\{X_t\}_{t \geq 0}\) be a (time-homogeneous) Itô diffusion in \(\mathbb{R}^d\).  
The (infinitesimal) generator \(\mathcal{A}\) of \(X_t\) is defined by
\[
\mathcal{A} f(x) = \lim_{t \to 0} \frac{\mathbb{E}^x[f(X_t)] - f(x)}{t}, \quad x \in \mathbb{R}^d.
\]
The set of functions \(f: \mathbb{R}^d \to \mathbb{R}\) such that the limit exists at \(x\) is denoted by \(D_{\mathcal{A}}(x)\), while \(D_{\mathcal{A}}\) denotes the set of functions for which the limit exists for all \(x \in \mathbb{R}^d\).
\end{definition}

\begin{theorem}{}{}
    Let \(\{X_t\}_{t \geq 0}\) be the Itô diffusion in \(\mathbb{R}^d\) described by the SDE

\[
\mathrm{d}X_t = b(X_t)\,\mathrm{d}t + \sigma(X_t)\,\mathrm{d}W_t.
\]
If \( f \in C_0^2(\mathbb{R}^d) \), then \( f \in D_{\mathcal{A}} \) and 

\[
\mathcal{A}f(x) = \sum_{i} b_i(x) \frac{\partial f}{\partial x_i} + \frac{1}{2} \sum_{i,j} (\sigma \sigma^T)_{i,j}(x) \frac{\partial^2 f}{\partial x_i \partial x_j},
\]
or, in terms of the gradient, scalar, and Frobenius inner products, the generator can be expressed as
\[
\mathcal{A}f(x) = b(x) \cdot \nabla_x f(x) + \frac{1}{2} \left( \sigma(x)\sigma(x)^\top \right) : \nabla_x \nabla_x f(x).
\]
\end{theorem}

For example, the generator \(\mathcal{A}\) for standard \(d\)-dimensional Brownian motion \(W\), which satisfies the stochastic diferential equation \(\mathrm{d}X_t = \mathrm{d}W_t\), is given by

\[
\mathcal{A} f(x) = \frac{1}{2} \sum_{i,j=1}^{d} \delta_{ij} \frac{\partial^2 f}{\partial x_i \partial x_j}(x) 
= \frac{1}{2} \sum_{i=1}^{d} \frac{\partial^2 f}{\partial x_i^2}(x)=\frac{1}{2}\Delta f(x),
\]
where \(\Delta\) denotes the Laplace operator.

Having defined the Itô diffusion process and its infinitesimal generator, we now turn to the study of the \textit{Kolmogorov backward equation}. This equation governs the evolution of the expected value of a given function of the process over time, providing a crucial link between stochastic processes and partial differential equations. The next theorem presents the formal statement of this equation.

\begin{theorem}{}{}
  Let \(\{X_t\}_{t \geq 0}\) be an Itô diffusion in \(\mathbb{R}^d\) satisfying the SDE
\[
\mathrm{d}X_t = b(X_t)\,\mathrm{d}t + \sigma(X_t) \,\mathrm{d}W_t. 
\] 
Let \(u: [0, \infty) \times \mathbb{R}^d \to \mathbb{R}\) be defined as
\[
u(t, x) = \mathbb{E}^x[f(X_t)],
\]
where \(f \in C_b^2(\mathbb{R}^d)\). Then, \(u(t, x)\) is differentiable with respect to $t$, \( u(t, \cdot) \in D_{\mathcal{L}}\), and $u$ satisfies the following partial differential equation, known as \textbf{Kolmogorov's Backward Equation}
\[
\begin{cases}
\frac{\partial u}{\partial t}(t, x) + \mathcal{L} u(t, x) = 0, & t > 0, \, x \in \mathbb{R}^d, \\
u(0, x) = f(x), & x \in \mathbb{R}^d,
\end{cases}
\]
where \(\mathcal{L}\) is the infinitesimal generator of the process \(X_t\).
\end{theorem}

The Kolmogorov backward equation describes the evolution of the expected value of a function of the process, conditioned on the present state. In contrast, the \textit{Fokker-Planck equation}, also known as the Kolmogorov forward equation, governs the time evolution of the \textit{probability density} of the process. While the backward equation concerns conditional expectations, the forward equation describes the evolution of the distribution itself. The following theorem formalizes this relationship, showing that the probability density of a diffusion process satisfies the Fokker-Planck equation, where the infinitesimal generator plays a central role in determining the evolution of the density.

\newpage

\begin{theorem}{}{}
Let \(\{X_t\}_{t \geq 0}\) be an Itô diffusion in \(\mathbb{R}^d\) satisfying the SDE
\[
\mathrm{d}X_t = b(X_t)\,\mathrm{d}t + \sigma(X_t) \,\mathrm{d}W_t.
\]
Assume that \(b(x) \in C^1( \mathbb{R}^d; \mathbb{R}^d)\) and \(\sigma(x) \in C^2(\mathbb{R}^n; \mathbb{R}^{d \times m})\). If \(p(t, x)\) is the probability density of \(X_t\), that is, 
\[ \mathbb{P}(X_t \in A) = \int_A p(t, x)\,\mathrm{d}x, \quad \forall A \in \mathcal{B}(\mathbb{R}^d), 
\]
and satisfies \( p(t, x) \in C^{1,2}((0, \infty) \times \mathbb{R}^d) \), with \( p_0(x) \in C^2(\mathbb{R}^d) \) and sufficient decay as \( \|x\| \to \infty \), that is, 
\[
p_0(x) \leq \frac{C}{(1 + \|x\|)^\alpha}, \quad \text{for some } C > 0, \alpha > 0.
\]
Then \( p(t, x) \) satisfies the \textbf{Fokker–Planck equation}
\[
\begin{cases}
\frac{\partial p}{\partial t}(t, x) = \mathcal{L}^* p(t, x), & t > 0, \, x \in \mathbb{R}^d, \\
p(0, x) = p_0(x), & x \in \mathbb{R}^d,
\end{cases}
\]
where the adjoint operator is
\[
\mathcal{L}^* p = - \sum_{i=1}^{n} \frac{\partial}{\partial x_i} \left( b_i(x) p \right) + \frac{1}{2} \sum_{i,j=1}^{n} \frac{\partial^2}{\partial x_i \partial x_j} \left( a_{ij}(x) p \right),
\]
where \(a(x) = \sigma(x) \sigma^\top(x)\) is the diffusion tensor. 
\end{theorem}

The relationship between the Brownian motion density and the heat equation, as shown in Corollary \ref{cor:BrownianHeatDensity}, establishes the foundation for the generalization of this result to diffusion processes. The Fokker-Planck equation extends this connection, governing the evolution of the probability density of a diffusion process under the influence of both deterministic drift and stochastic fluctuations. While the heat equation describes the time evolution of the Brownian motion density, the Fokker-Planck equation encompasses a broader class of stochastic processes, capturing their complex dynamics through the infinitesimal generator.

\begin{figure}[ht]
    \centering
    \includegraphics[width=0.95\textwidth, height=0.25\textheight]{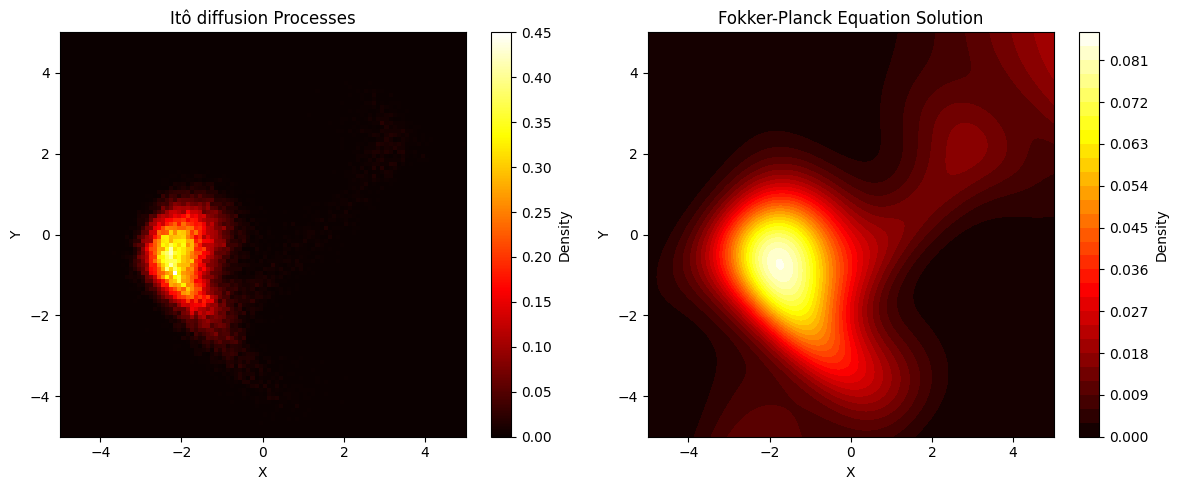}
    \caption[Comparison of empirical density of a stochastic diffusion process and the theoretical Fokker–Planck solution with asymmetric drift and spatially dependent diffusion.]{Comparison between the empirical density of a stochastic diffusion process (left) and the theoretical solution of the Fokker–Planck equation (right), with asymmetric drift and spatially dependent diffusion.}
    \label{fig:Conexion-FPE}
\end{figure}

We numerically illustrate the relationship between the \textit{empirical density} of a stochastic diffusion process and the \textit{theoretical solution} of the Fokker–Planck equation, extending the classical connection between Brownian motion and the heat equation. Consider the Itô diffusion process \(\{X_t\}_{t \geq 0} \subset \mathbb{R}^2\) governed by 
\[
\,\mathrm{d}X_t = b(X_t) \,\mathrm{d}t + \sigma(X_t)\,\mathrm{d}W_t,
\]
where \(b(x,y)\) and \(\sigma(x,y)\) denote the {drift field} and {diffusion coefficient}, respectively. The system exhibits \textit{asymmetric drift} and \textit{spatially dependent diffusion}, defined as  
\[
b_x(x,y) = -0.3x + 1.5\sin(y), \quad b_y(x,y) = -0.3y - 1.5\cos(x),
\]
\[
\sigma_x(x,y) = 0.3 + 0.2|\sin(x)|, \quad \sigma_y(x,y) = 0.3 + 0.2|\cos(y)|.
\]
The stochastic trajectories are simulated using the \textit{Euler–Maruyama scheme} with \( N = 40,000 \) particles over \( T = 3.0 \), discretized into \( N_t = 500 \) time steps. Figure~\ref{fig:Conexion-FPE} compares the \textit{empirical density} (left panel), estimated from particle trajectories, with the \textit{numerical solution} of the Fokker–Planck equation (right panel). The agreement between both representations confirms that the Fokker–Planck equation effectively describes the \textit{macroscopic evolution} of the probability density while averaging out stochastic fluctuations present in the empirical distribution. 

\vspace{0.3cm}

The \textit{Feynman-Kac formula} generalizes the Kolmogorov backward equation by incorporating both \textit{potential fields} and \textit{terminal conditions} into the description of the diffusion process. While the Kolmogorov backward equation focuses on the evolution of conditional expectations, the Feynman-Kac formula extends this framework to systems influenced by external potentials and subject to specific boundary conditions. The following theorem formalizes this connection and introduces the Feynman-Kac formula for Itô diffusion processes.

\newpage

\begin{theorem}{Feynman-Kac Formula}
Let \(\{X_t\}_{t \geq 0}\) be an Itô diffusion process in \(\mathbb{R}^d\) governed by the SDE
\[
\,\mathrm{d}X_t = b(X_t) \,\mathrm{d}t + \sigma(X_t)\,\mathrm{d}W_t,
\]
where \( b: \mathbb{R}^d \to \mathbb{R}^d \) is the drift field, \(\sigma: \mathbb{R}^d \to \mathbb{R}^{d \times m} \) is the diffusion coefficient, and \(\{W_t\}_{t \geq 0}\) is a Brownian motion in \(\mathbb{R}^d\). Define the function \( u: [0,T] \times \mathbb{R}^d \to \mathbb{R} \) as the conditional expectation
\[
u(t,x) = \mathbb{E} \left[ e^{-\int_t^T V(X_s) ds} g(X_T) \mid X_t = x \right],
\]
where \( V: \mathbb{R}^d \to \mathbb{R} \) is a potential function and \( g: \mathbb{R}^d \to \mathbb{R} \) is a terminal condition. 

Then, if \( u \in C^{1,2}([0,T] \times \mathbb{R}^d) \), it satisfies the following Kolmogorov-type partial differential equation
\[
\frac{\partial u}{\partial t} + \mathcal{L} u - V u = 0, \quad (t,x) \in [0,T) \times \mathbb{R}^d,
\]
with terminal condition
\[
u(T, x) = g(x), \quad x \in \mathbb{R}^d.
\]
Here, \(\mathcal{L}\) denotes the infinitesimal generator of the Itô diffusion \( X_t \), given by
\[
\mathcal{L} u = \sum_{i=1}^{d} b_i(x) \frac{\partial u}{\partial x_i} + \frac{1}{2} \sum_{i,j=1}^{d} a_{ij}(x) \frac{\partial^2 u}{\partial x_i \partial x_j},
\]
where \( a(x) = \sigma(x) \sigma^{\top}(x) \) is the diffusion tensor.
\end{theorem}

\section{Conclusions}

In this lecture notes, we have established the theoretical foundations necessary for understanding the interplay between stochastic processes and partial differential equations. We began by examining the classical relationship between Brownian motion and the heat equation, which illustrates how microscopic random motion induces a macroscopic deterministic evolution. This fundamental connection serves as the entry point for a broader theory. To extend this perspective to more general stochastic systems, we introduced Itô calculus, which provides the rigorous framework for defining stochastic integrals and for analyzing processes with both deterministic drift and stochastic noise. Equipped with these tools, we then developed the formal theory of diffusion processes, characterized as strong solutions to stochastic differential equations. We concluded the work by exploring how the macroscopic behavior of diffusion processes is governed by the Fokker–Planck equation, a generalization of the heat equation, and how both the Kolmogorov backward equation and the Feynman–Kac formula link expectations of stochastic dynamics with solutions of associated PDEs. 

\section*{Interactive Notebook Access}
For replication and further exploration, you can access the interactive Google Colab notebook accompanying these lecture notes via the following link:

\begin{center}
\href{https://colab.research.google.com/drive/161byYNgy2z4JzADjQFQbA72CfHat5Xyx?usp=sharing}{\texttt{\textcolor{red!70!black}{Open Notebook on Google Colab}}}
\end{center}

\bibliographystyle{plain}  
\bibliography{bibliography} 

\begin{thebibliography}{1}

\bibitem{evans2010pde}
Lawrence~C. Evans.
\newblock {\em Partial Differential Equations}, volume~19 of {\em Graduate Studies in Mathematics}.
\newblock American Mathematical Society, 2nd edition, 2010.

\bibitem{Kallenberg2021}
Olav Kallenberg.
\newblock {\em Foundations of Modern Probability}, volume~99 of {\em Probability Theory and Stochastic Modelling}.
\newblock Springer, 3rd edition, 2021.

\bibitem{kuo}
H.~H. Kuo.
\newblock {\em Introduction to Stochastic Integration}.
\newblock Springer, New York, 2006.

\bibitem{oksendal}
B.~{\O}ksendal.
\newblock {\em Stochastic Differential Equations: An Introduction with Applications}.
\newblock Springer, Berlin, 2003.

\bibitem{Protter2005}
Philip~E. Protter.
\newblock {\em Stochastic Integration and Differential Equations}, volume~21 of {\em Stochastic Modelling and Applied Probability}.
\newblock Springer Berlin, Heidelberg, 2nd edition, 2005.

\bibitem{Rogers_Williams_2000}
L.~C.~G. Rogers and David Williams.
\newblock {\em Diffusions, Markov Processes and Martingales}.
\newblock Cambridge Mathematical Library. Cambridge University Press, 2 edition, 2000.

\end{thebibliography}

\end{document}